\documentclass[a4paper,leqno,12pt, reqno]{amsart}

\usepackage{amsmath}
\usepackage{amssymb}
\usepackage{graphicx}
\usepackage{tikz-cd}
\usepackage{hyperref}

\usepackage{xspace}
\usepackage{bm}
\usepackage{amsmath}
\usepackage{amstext}
\usepackage{amsfonts}
\usepackage[mathscr]{euscript}
\usepackage{amscd}
\usepackage{latexsym}
\usepackage{amssymb}
\usepackage{enumerate}
\usepackage{xcolor}

\theoremstyle{plain}
\swapnumbers

 \newtheorem{theorem}{Theorem}[section]
    \newtheorem{prop}[theorem]{Proposition}

    \newtheorem{lemma}[theorem]{Lemma}
    \newtheorem{corollary}[theorem]{Corollary}
       
     \newtheorem*{thma}{Theorem A}
    \newtheorem*{thmb}{Theorem B}
    \newtheorem*{thmc}{Theorem C}

\theoremstyle{definition}
    \newtheorem{definition}[theorem]{Definition}
    \newtheorem{example}[theorem]{Example}

\theoremstyle{remark}

\newcommand{\ccn}[1]{\mathcal{CC}_{{#1}}}

\newcommand{\ccnm}[2]{\mathcal{CC}_{{#1}, {#2}}}
\newcommand{\ccdnm}[2]{\mathcal{CC'}_{{#1}, {#2}}}
\newcommand{\ccdn}[1]{\mathcal{CC'}_{{#1}}}

\newcommand{\cckn}[2]{\mathcal{CC}^{#1}_{#2}}
\newcommand{\ccdkn}[2]{\mathcal{CC'}^{#1}_{ #2}}

\newcommand{\sd}{\operatorname{sd}}
\newcommand{\Id}{\operatorname{Id}}

\begin{document}

\title{ Higher analogs of simplicial and combinatorial complexity}

\author{Amit Kumar Paul}

\address{Department of Mathematics and Statistics\\
Indian Institute of Technology, Kanpur\\ Uttar Pradesh 208016\\India}
\email{kamitp@iitk.ac.in}

\date{\today}


\begin{abstract}
We introduce higher simplicial complexity of a simplicial complex $K$ and higher combinatorial complexity of a finite space $P$ (i.e. $P$ is a finite poset). We relate higher simplicial complexity with higher topological complexity of $|K|$ and higher combinatorial complexity with higher simplicial complexity of the order complex of $P$. 
\end{abstract}
\maketitle

\begin{section}{Introduction}
 The \emph{topological complexity } $TC(X)$ of a path connected space $X$ was introduced by Farber (see \cite{far}).   It is a measure of the complexity to construct a motion-planning algorithm on the space $X$. Let $I = [0,1]$ and $PX =X^I $ denotes the free path space. Consider the fibration

\begin{equation}\label{pi}
  \pi : PX \rightarrow X\times X,~~ \gamma \mapsto (\gamma(0), \gamma(1)).
 \end{equation}
 
\noindent Then $TC(X)$ is defined to be the least positive integer $r$ such that there exists an open cover $\{U_1, \cdots, U_r\}$ of $X\times X$ with continuous section of $\pi$ over each $U_i$ (i.e. a continuous map $s_i : U_i \rightarrow E$ satisfying $\pi \circ s_i = \Id_{U_i}$ for $i = 1, 2, \cdots, r$).  The idea was generalised by Rudyak to higher dimensions (see \cite{rud}). He introduced \emph{n-th topological complexity } $TC_n(X), \ n \geq 2$  such that $TC_2(X) = TC(X)$. We recall the definition of higher topological complexity in the next section.

In (\cite{gon}), Gonzalez used contiguity of simplicial maps to define \emph{simplicial complexity} $SC(K)$ for a simplicial complex $K$. This is a discrete analogue of topological complexity in the category of simplicial complexes.  He showed that $SC(K) = TC(\mid K \mid)$ for finite $K$, where $\mid K \mid$ is the geometric realization of $K$.
We introduce \emph{higher simplicial complexity} $SC_n(K)$ of a simplicial complex $K$ and generalise the above result. 
\begin{thma}

For a finite simplicial complex $K$,  $SC_n(K) = TC_{n}(\mid K\mid)$ for any $n \geq 2$.
\end{thma}
\noindent (See Theorem \ref{tst}.)

 A combinatorial approach to topological complexity was introduced by Tanaka (cf. \cite{tana}). The basic idea of  Tanaka's paper is to describe topological complexity by combinatorics of finite spaces  i.e. connected finite $T_0$ space ( see \cite{sto}). He used an analogue of the above path-space fibration for finite spaces  to define  \emph{combinatorial complexity} $CC(P)$. It is shown that $CC(P) = TC(P)$.  
 We introduce an analogue  \emph{higher combinatorial complexity} $\cckn{}{n}(P)$ and prove the above result for higher dimensions.
   \begin{thmb}
 For  any finite space $P$, we have $\cckn{}{n}(P) = TC_n(P) $, for any  $n \geq 2$. 
\end{thmb}
(See Theorem \ref{tctp})

 Given a finite space $P$, there is a naturally associated simplicial complex, called  \emph{order complex} $\mathcal{K}(P)$. A finite space is equivalent to a finite poset and the $n$-simplices of $\mathcal{K}(P)$ are linearly ordered subsets of $P$   (c.f \ref{oc}). 
 As noted by Tanaka, $CC(P)$ is just an upper bound to to $SC(\mathcal{K}(P))$.  To describe $SC(\mathcal{K}(P))$ combinatorially, he used barycentric subdivision of $P$ to define $CC^\infty(P)$. Finally it is shown that   $CC^\infty(P) = SC(\mathcal{K}(P))$. Hence $CC^\infty(P) = TC(\mid \mathcal{K}(P) \mid)$.
 
  We further generalise above ideas to \emph{higher combinatorial complexities} $\cckn{\infty}{n}(P)$ using barycentric subdivision  of a finite space $P$ and prove the prove following.

\begin{thmc}
 For  any finite space $P$, we have $\cckn{\infty}{n}(P) = SC_n( \mathcal{K}(P) )$, for any  $n \geq 2$. 
\end{thmc}
(See Theorem \ref{tsc})

The organization of the rest of paper is as follows. In Section 2 we recall basics of topological complexity and higher topological complexity of a space $X$. In Section 3 we  introduce higher simplicial complexity of a simplicial complex $K$ and prove Theorem (A). In Section 4 we define higher combinatorial complexity $\ccn{n}(P)$ of a finite space $P$ and prove Theorem (B). In 5 we describe higher combinatorial complexity with barycentric subdivision $\cckn{\infty}{n}(P)$ of a finite space $P$ and we prove Theorem (C).

\end{section}
\begin{section}{Preliminaries}

In this section we review basic concepts of topological complexity. For details we refer to \cite{farbook, far, rud}. We start by recalling the definition of the \emph{Schwarz genus} of a fibration. Let $p : E \rightarrow B$ be a fibration. The \emph{Schwarz genus} of $p$ is the minimum number $k$ such that $B$ can be cover by $k$ open subsets, $U_1 \cup U_2 \cup \cdots \cup U_k = B$ and on each $U_i$ there is a section of $p$. It is denoted by \emph{genus}$(p)$. If no such $k$ exists then we say $genus(p) =\infty$. Then the topological complexity of $X$ is \emph{genus}$(\pi)$, where $\pi$ is as in the equation \ref{pi}.

 Suppose $I_n, n \geq 2$ denote the wedge of $n$ intervals $[0, 1]_j, j = 1, 2, \cdots, n$, where $0_j \in [0, 1]_j$ are identified. Consider the mapping space $X^{I_n}$ and the fibration
$$e_n : X^{I_n} \rightarrow X^n ,~~~e_n(\alpha) = (\alpha(1_1), \alpha(1_2),\cdots,\alpha(1_n)).$$ 

\noindent The \emph{n-th topological complexity} $TC_n(X)$ of $X$ is defined to be \emph{genus}$(e_n)$. It can be defined alternatively as  $TC_n(X) = \emph{genus}(e'_n)$, where $$e'_n : X^{I}=PX \rightarrow X^n ,  ~~~~e'_n(\alpha) = (\alpha(0), \alpha(\frac{1}{n-1}), \alpha(\frac{2}{n-1}),\cdots,\alpha(1)).$$ 

\noindent Note that $TC_2(X)$ is nothing but $TC(X)$.  It is proved that $TC_n(X)$ is homotopy invariant and $TC_n(X) = 1$ if and only if $X$ is contractible. The sequence $TC_n(X)$ is an increasing sequence i.e, $TC_n(X) \leq TC_{n+1}(X)$. Topological complexity of a space $X$ is closely related to the \emph{Lusternik–Schnirelmann category} (or L-S category) of the space $X$, which is denoted by $cat(X)$. Recall that $cat(X)$ is defined as $cat(X) = \emph{genus}(\pi_0)$, where $\pi_0 : P_0X \rightarrow X$ given by $\pi_0(\gamma) = \gamma(1)$ and $P_0X$ is the space of all  paths in $X$ with a fixed starting point $x_0$. The topological complexity and L-S category of a space satisfy the following inequality:

\begin{center}
$$cat(X) \leq TC_{n}(X) \leq cat(X^n) \leq cat(X)^n.$$
\end{center} 

We will use the following Lemma to define higher simplicial complexity. It is simple generalisation of \cite[Lemma  1.1 ]{gon}

\begin{lemma}\label{lemm}

The evaluation map $e'_n : PX \rightarrow X^n$ admits a section on a subset $A$ of $X^n$ if and only if the composition maps $\pi_1, \pi_2,\cdots\pi_n : A \hookrightarrow X^n \rightarrow X$ are in same homotopy class of maps.

\end{lemma}

\begin{proof}

Let $e'_n$ admits a section $s : A \rightarrow PX$ on a subset $A$ of $X^n$. Let $H : A \times I \rightarrow X$ be the map given by $H(x_1, x_2, \cdots, x_n, t) = s(x_1, x_2, \cdots x_n)(t)$. Then clearly $H|_{A\times [\frac{j-1}{n-1}, \frac{j}{n-1}]} : \pi_j \simeq \pi_{j+1}$ for $j=1,\cdots, n-1$. Conversely, assume that $\pi_1, \pi_2, \cdots \pi_n$ are in same homotopy class and $h^j_t : A \rightarrow X$ be a homotopy from $\pi_j$ to $\pi_{j+1}$ i.e, $h^j_0(x_1, x_2, \cdots, x_n) = x_j$ and $h^j_1(x_1, x_2, \cdots, x_n) = x_{j + 1}$ for $j = 1, 2, \cdots {n-1}$. Then the concatenation of the paths $h^1_t, h^2_t, \cdots, h^{n-1}_t$ will give a section on $A$.
\end{proof}

 We will use the following Proposition to relate $n$-th simplicial complexity of a simplicial complex $K$ and $n$-th topological complexity of the geometric realization of $K$. This is a simple generalisation of the result of \cite[Proposition 4.12 and Remark 4.13]{farbook}

\begin{prop}\label{pro}
Let $X$ be an ENR. Then $TC_n(X) = r$, where $r$ is the minimal integer such that there exist a section $s : X^n \rightarrow PX$ (which is not necessarily continuous) of the fibration $e'_n$ and a splitting $ G_1 \cup G_2 \cup \cdots \cup G_r = X^n$, where each $G_i$ is locally compact subset of $X^n$ and each restriction $s_{|G_i} : G_i \rightarrow PX$ is continuous for $i = 1, 2,\cdots,r$.
\end{prop}
\begin{proof}
Let us assume that $TC_n(X) = k$. Consider an open cover $\{U_i\}_{i=1}^k$ of $X^n$ such that on each open set there is a continuous section $s_i : U_i \rightarrow PX$ of $e'_n$. Set $G_1 = U_1$ and $G_i = U_i - ( U_1 \cup U_2 \cup \cdots \cup U_{i-1} )$ for $i \geq 2$. Then $G_i$'s are locally compact and cover $X^n$. Let $s = s_1 \cup s_{2|_{G_2}} \cup s_{3|_{G_3}} \cdots \cup s_{k|_{G_k}}$. Then $s : X^n \rightarrow PX$ is a section of $e'_n$ which is continuous on each $G_i$. Thus $r \leq k = TC_n(X)$. 

Conversely, suppose that $G_1 \cup G_2 \cup \cdots \cup G_{r} = X^n$  where $\{G_i\}$  are locally compact subsets. Assume that $s : X^n \rightarrow PX$ is a section (which is not necessarily continuous) of the fibration $e'_n$ such that $s|_{G_i} = s_i$ are continuous. These $s_i$'s are in one-to-one correspondence with homotopies $h^i_t : G_i \rightarrow X$, $t \in [0, 1]$ such that $h^i_0, h^i_{\frac{1}{n-1}}, h^i_{\frac{2}{n-1}}, \cdots, h^i_1 : G_i \rightarrow X$ are the $n$ projections $X^n \to X$ restricted to $G_i$. Using the same argument as  in \cite[Proposition 4.12(c)]{farbook}, we can extend the section $s_i$ on an open subset containing $G_i$. Hence $TC_n(X) \leq r.$

\end{proof}
\end{section}

\begin{section}{Higher simplicial complexity}

 A simplicial approach to topological complexity was introduced by Gonzalez's (\cite{gon}). He introduced the notion of simplicial complexity $SC(K)$ for simplicial complex $K$. This was based on 
contiguity class of simplicial maps. It is proved in (\cite{gon}) that simplicial complexity $SC(K)$  is equal to the topological complexity $TC( \mid K \mid )$ of geometric realization of $K$, for a finite simplicial complex $K$. In this section we introduce higher analog of simplicial complexity $SC_{n}(K)$ and prove that for a finite simplicial complex $K$, $SC_{n}(K) = TC_{n}( \mid K \mid )$.

 We begin by recalling the definition and some basic facts about contiguity of simplicial maps. 
For a positive integer $c$, two simplicial maps $\phi , \phi' : K \rightarrow L$ are called $c$-\emph{contiguous} if there is a sequence of simplicial maps $\phi = \phi_0, \phi_1, \phi_2 \cdots, \phi_c = \phi' : K \rightarrow L$, such that $\phi_{i-1}(\sigma) \cup \phi_{i}(\sigma)$ is a simplex of $L$ for each simplex $\sigma$ of $K$ and $i \in \{1, 2,\cdots,c\}$.  We write $\phi \sim \phi'$ if  $\phi$ and $\phi'$ are $c$-contiguous for some positive integer $c$. This defines an equivalence relation on the set of simplicial maps $K \rightarrow L$ and the equivalence classes are called \emph{contiguity classes}. Simplicial maps in the same contiguity class have homotopic topological realization.
  
 We denote barycentric subdivision of $K$ by $\sd(K)$ and $\sd^{k+1}(K)= \sd(\sd^k(K)).$  We choose a \emph{simplicial approximation} of the identity on $\mid K \mid \times \mid K \mid,$  $$\iota_{\sd^{k-1}(K \times K)} : \sd^{k}(K \times K) \rightarrow \sd^{k-1}( K \times K),~~ k \geq 1.$$  The iterated compositions are denoted by $$\iota^{k}_{K \times K} : \sd^{k}(K \times K) \rightarrow  K \times K,~~ k \geq 0.$$ 
Let $\pi_{j} : \sd^{k}(K \times K) \rightarrow K$ denote the composition of $\iota^{k}_{K \times K} $ and $j^{th}$ projection for $j = 1, 2$.

\noindent Let $K$ be a finite simplicial complex. In \cite{gon} Gonzalez defined $SC^{k}(K)$ to be the smallest nonnegative integer $r$ such that there exist subcomplexes $\{L_i\}_{i=1}^{r}$ covering $\sd^{k}(K \times K)$ and the restrictions $\pi_1 , \pi_2 : L_i \rightarrow K$ lie in the same contiguity class for each $i$. 
%Note that in (\cite{tana}) Tanaka define the above definition and in (\cite{gon}) Gonzalez's used the reduced version, which is one less than the above definition.
Then  $SC^{0}\geq SC^{1} \geq SC^{2} \geq \cdots \geq 1$.
The \emph{simplicial complexity} $SC(K)$ is defined as the minimum of the $SC^{k}(K)$ i.e.,

$$SC(K) = \lim_{k \to \infty} SC^{k}(K) = \min_{k \geq 0}\{SC^{k}(K)\}.$$

 Now we introduce \emph{higher simplicial complexity} $SC_n(K)$ of simplicial complex $K$. As the previous case we choose a simplicial approximation $\iota_{\sd^{k-1}( K^n)} : \sd^{k}(K^n) \rightarrow \sd^{k-1}( K^n)$ of the identity on $\mid K \mid ^n$ for $k \geq 0$. We denote iterated compositions by $\iota^{k}_{K^n} : \sd^{k}(K^n) \rightarrow  K^n $ and $\pi_{j} : \sd^{k}(K^n) \rightarrow K$ denote the composition of $\iota^{k}_{K^n}$ and $j^{th}$ projection $K^n \to K$ for $j = 1, 2,\cdots,n$.
 
\begin{definition}
Let $K$ be a simplicial complex. We define $SC^{k}_n(K)$ as the smallest nonnegative integer $r$ such that there exist subcomplexes $\{L_i\}_{i=1}^{r}$ covering $\sd^{k}(K^n)$ and the restrictions $\pi_j : L_i \rightarrow K$, for $j = 1, 2,\cdots,n$ lie in the same contiguity class, for each $i$. If no such $r$ exists then we define $SC^{k}_n(K)$ to be $\infty$.
\end{definition}

\noindent It is to be noted that the value $SC^{k}_{n}$ is independent of the chosen approximation $\iota^{k}_{K^n} : \sd^{k}(K^n) \rightarrow  K^n$ of identity on $\mid K \mid^n$.  As in the simplicial complexity of Gonzales,  we have $SC_n^{0}\geq SC_n^{1} \geq SC_n^{2} \geq \cdots \geq 1$ for any simplicial complex $K$. Using the Proposition (\ref{pro}) we deduce the following:

\begin{lemma}\label{lem}
For a simplicial complex $K$, we have the following inequalities: 

\begin{enumerate}[(i)]
\item  $SC_n^k(K) \geq TC_{n}(\mid K \mid)$ for any $n \geq 2$.
\item $SC_n^k(K) \geq SC_n^{k+1}(K)$ for all $n\geq 2,~~ k \geq 0$.
\end{enumerate}
\end{lemma}
 
 \begin{proof}
(i) If we first apply the geometric realization functor of simplicial complexes, then $\mid \pi_j \mid$ are homotopic, for $ j = 1, 2, \cdots, n$.  Then use Lemma \ref{lemm} to get sections $s_i : \mid L_i \mid \rightarrow P\mid K \mid$ over locally compact susbets $\mid L_i \mid$. If we set $G_1 = \mid L_1 \mid$ and $G_i = \mid L_i \mid - (\mid L_1 \mid \cup \mid L_2 \mid \cup \cdots \cup \mid L_{i-1} \mid )$ for $i \geq 2$, then each $G_i$ is also locally compact. Define $s : \mid K \mid ^n \rightarrow  P\mid K \mid$, is the union of $s_i$ restricted over $G_i$ and lastly we apply Proposition \ref{pro} to conclude $SC_n^k(K) \geq TC_{n}(\mid K \mid)$ for any $n \geq 2$. 

 (ii) Let $J$ be a subcomplex of $\sd^k(K^n)$ on which $\pi_1, \pi_2,\cdots,\pi_n$ are in same contiguity class. Assume that $\lambda_J : \sd(J) \rightarrow J$ is an approximation of identity on $J$. Obviously $\sd(J)$ is a subcomplex of $\sd^{k+1}(K^n)$. We will show that  $\pi_1, \pi_2,\cdots,\pi_n$ are in same contiguity class on $\sd(J)$. 
The two compositions of the diagram are contiguous.

\begin{center}
$\begin{tikzcd}
J\arrow[hookrightarrow]{r} & \sd^k(K^n)
 \\
  \sd(J)\arrow{u}{\lambda}\arrow[hookrightarrow]{r} & \sd^{k+1}{K^n}\arrow{u}{\iota}
\end{tikzcd}$
\end{center}
So, $\pi_1, \pi_2,\cdots,\pi_n$ are in same contiguity class on $\sd(J)$.
 \end{proof}
 
 Since $\{SC_n^{k}(X)\}_{k=1}^{\infty}$ is a decreasing sequence of integers, we can make the following definition. 
\begin{definition}
For a finite simplicial complex $K$, the \emph{$n$-th simplicial complexity} $SC_n(K)$ is defined as the minimum of the $SC_n^{k}(K)$:

$$SC_n(K) = \lim_{k \to \infty} SC_n^{k}(K) = \min_{k \geq 0}\{SC_n^{k}(K)\}.$$

\end{definition}

We now prove the main Theorem of this section. The proof is analogous to \cite[Theorem 3.5]{gon}. The following result is required in the proof (\cite[Chapter 3]{span}). 
 
 \begin{prop} \label{prop}
 Let $f_1, f_2,\cdots,f_n : \mid K \mid \rightarrow \mid L \mid$ be continuous maps, which are in same homotopic class, then there is $k_0 \in \mathbb{N}$ such that, for each $k \geq k_0$ and any approximation $\phi_1, \phi_2,\cdots,\phi_n : \sd^k(K) \rightarrow L$ of  $f_1, f_2,\cdots,f_n$ respectively, are belongs to same contiguity class.

\end{prop}

\begin{theorem}\label{tst}
For a finite simplicial complex $K$,  $SC_n(K) = TC_{n}(\mid K\mid)$ for any $n \geq 2$.
\end{theorem}

\begin{proof}
From the Lemma (\ref{lem}) one can say that $ SC_n(K) \geq TC_{n}(\mid K \mid)$. We now prove  the other inequality.  Assume that $TC_n(\mid K \mid) = r$. We choose a motion planner $\{(U_1, s_1), (U_2, s_2),\cdots,(U_r, s_r)\}$ for $\mid K \mid$. Using the finiteness assumption on $K$, we choose a large positive integer $k$ so that the realization of each simplex of $\sd^k(K^n)$ is contained in some $U_i ~ (0 \leq i \leq r)$. For each $i \in \{1, 2,\cdots,r\}$ let $L_i$ be the subcomplex of $\sd^k(K^n)$ consisting of those simplices whose realization are contained in $U_i$. Then $L_1, L_2,\cdots,L_r$ covers $K^n$. Now the projections $p_{1}, p_{2},\cdots ,p_{n} : \mid K \mid ^{n} \rightarrow \mid K \mid $ belong to same homotopy class over each $U_i$ and, in particular, over the realization of the corresponding subcomplex $L_i$. Therefore by Proposition (\ref{prop}) there is a positive integer $k'$ such that, for each $i \in \{1, 2,\cdots,r\}$ the compositions $\sd^{k'}(L_i) \hookrightarrow \sd^{k+k'}(K^n) \rightarrow K$ are belongs to same contiguity class. Hence $SC_n(K) \leq TC_n(\mid K \mid)$, and thus $SC_n(K) = TC_{n}(\mid K\mid)$.
\end{proof}
\end{section}

\begin{section}{Higher combinatorial complexity}

 In this section we introduce the higher analogue of combinatorial complexity of a finite space, as defined by Tanaka in \cite{tana}.  We refer reader to \cite{sto} for finite spaces. A \emph{finite space} $P$ is a finite $T_0$ space. For any $x \in P$ we denote $U_x$ be the intersection of all open set  containing $x$. Now define a partial relation on $P$ by $x \leq y$ if and only if $U_x \subseteq U_y$. So we can consider $P$ as a poset. On the other hand, given a finite  poset, we have a $T_0$ finite set with $U_x = \{y: ~~y\leq x\}$. Thus a finite space is equivalent to a finite poset.  From now onwards we assume all our finite spaces are connected. 
 
 A map between finite spaces is continuous if and only if it preserves the partial order.  Let $J_m$ denote the finite space consisting $m+1$ points with the zigzag order 
\begin{center}
$0 \leq 1 \geq 2 \leq \cdots \geq (\leq)m$.
\end{center}
This finite space is called the \emph{finite fence} with length $m$. It behaves like an interval in the category of finite spaces. An order preserving map $\gamma : J_m \rightarrow P$ is called a combinatorial path or simply a path. Thus a combinatorial path is just a zigzag $\gamma(0) \leq \gamma(1) \geq \gamma(2) \leq \cdots \geq (\leq) \gamma(m)$ of elements of $P$.
  A conneted finite space is always path connected. If $\gamma_1$ and $\gamma_2$ be two combinatorial paths in $P$ of length $m_1$ and $m_2$ respectively, satisfying $\gamma_1 (m_1) = \gamma_2 (0)$. Then the concatenation of $\gamma_1$ and $\gamma_2$ is a path $\gamma_1\ast\gamma_2: J_m \to P$ where $m = m_1 + m_2$ or $m_1 + m_2 + 1$ according to $m_1$ is even or odd. It is define as: 
  
\begin{center}
$\gamma_1*\gamma_2 (i) = \begin{cases}
\gamma_1 (i) & \mbox{ if } ~~ 0 \leq i \leq m_1\\
\gamma_2 (i-m_1) & \ \mbox{if } ~~ m_1 \leq i \leq m_1+ m_2 .
 \end{cases}$
 \hspace{1.7 cm} $(m_1$ -even)
 \end{center}

 \begin{center}
$\gamma_1*\gamma_2 (i) = \begin{cases}
\gamma_1 (i) & \mbox{ if } ~~0\leq  i \leq m_1\\
\gamma_1 (i-1) & \mbox{ if } ~~ i = m_1 + 1\\
\gamma_2 (i-m_1-1) & \ \mbox{if } ~~m_1+1\leq  i \leq m_1 +m_2+ 1 .
 \end{cases}$
 $(m_1$ -odd)
\end{center} 

\noindent Two maps $f, g : P \rightarrow Q$ between two finite spaces are called homotopic if there exist $m \geq 0$ and a continuous map (or an order preserving map) $H : P \times J_m \rightarrow Q$ such that $H(x, 0) = f(x)$ and $H(x, m) = g(x)$.  The finite space of all combinatorial paths of $P$ with length $m$, equipped with the pointwise order, is denoted by $P^{J_m}$. As an analog of path fibration, it is equipped with the canonical order preserving map 
$$q_m: P^{J_m} \rightarrow P \times P \text{ defined by } q_m(\gamma) = (\gamma (0), \gamma (m)), ~~m \geq 0$$ 
 
 \noindent In (\cite{tana}) Tanaka defined $CC_{m}(P)$   to be the smallest nonnegative integer $r$ such that there exist an open cover of $P \times P$ consisting $r$ open sets with a section  of $q_m$ on each open set. He proved that $CC_{m}(P)$ is decreasing sequence on $m$ and defined $CC(P)$ be the limit of $CC_{m}(P)$. Also he proved that $CC(P) = TC(P)$. In this section we will generalise the above idea.
 
 Let $n\geq 2$ and $J_{n,m}$ be the finite set of $nm + 1$ points $$\{0, 1_1, 1_2,\cdots ,1_n, 2_1, 2_2,\cdots ,2_n,\cdots ,m_1, m_2,\cdots,m_n\}.$$  The partial ordering on $J_{n,m}$  consists of $n$ finite fances each length $m$ as below:

 $$0 \leq 1_1 \geq 2_1 \leq \cdots \geq (\leq) m_1,$$
 
 $$0 \leq 1_2 \geq 2_2 \leq \cdots \geq (\leq) m_2,$$
 
  $$\cdots $$
  
  $$0 \leq 1_n \geq 2_n \leq \cdots \geq (\leq) m_n.$$

\noindent Consider the space $P^{J{n,m}}$ is the space of all order preserving map $J_{n,m} \rightarrow P$. We have a canonical projection $$q_{n, m} : P^{J_{n, m}} \rightarrow P^n, ~~q_{n,m}(\gamma) = (\gamma (m_1), \gamma(m_2),\cdots , \gamma (m_n)). $$ We define $\ccnm{n}{m}(P)$ as the smallest positive integer $r$ such that $P^n$ can be cover by $r$ open sets $\{Q_i\}_{i=1}^r$ with section $s_i : Q_i \rightarrow P^{J_{n, m}}$ of $q_{n, m}$ for each $i$. If no such $r$ exist then we set $\ccnm{n}{m}(P) = \infty$.  The following lemma shows that $\ccnm{n}{m}$  decreases as we increase $m$.  

 \begin{lemma}\label{ccnm}
 For any $m \geq 0, n\geq 2$, it holds that $\ccnm{n}{m+1}(P) \leq \ccnm{n}{m}(P)$.
 \end{lemma}
 
 \begin{proof}
 Let $\ccnm{n}{m}(P) = r$ and  $\{Q_i\}_{i=1}^r$ be an open cover of $P^n$ with section $s_i : Q_i \rightarrow P^{J_{n, m}}$. Consider the retraction map $R : J_{n, m+1} \rightarrow J_{n, m}$ sending each $(m+1)_i$ to $m_i$ for $1\leq i\leq n$. This is clearly an order preserving map. It will induce a map $R^* : P^{J_{n, m}} \rightarrow P^{J_{n, m+1}}, \gamma \rightarrow \gamma \circ R$ such that the following diagram commutes:
 \begin{center}
 $ \begin{tikzcd}
P^{J_{n, m}} \arrow{rr}{R^{*}} \arrow[swap]{dr}{q_{n, m}}& &P^{J_{n, m+1}} \arrow{dl}{q_{n, m+1}}\\
& P^n & 
\end{tikzcd}$
 \end{center}
 The composition $R^* \circ s_i : Q_i \rightarrow P^{{n, m+1}}$ is a section of $q_{n, m+1}$ for each $i$. Thus, $\ccnm{n}{m+1}(P) \leq r$.
 \end{proof} 
 
 \begin{definition} \label{dccn}
 For a finite space $P$ we define $\ccn{n}(P)$ is the minimum of the $\ccnm{n}{m}(P)$
 \begin{center}
 $$\ccn{n}(P) = \min_{m \geq 1}\{ \ccnm{n}{m}(P)\} = \lim_{m \rightarrow \infty}\ccnm{n}{m}(P).$$ 
 \end{center}
 \end{definition}

 We now give an alternative description of $\ccn{n}(P)$. Later we will use both the description interchangably. Consider the space $P^{J_{(n-1)m}}$ and the projection map $$q'_{n,m} : P^{J_{(n-1)m}} \rightarrow P^n,~~ q'_{n,m}(\alpha) = (\alpha (0), \alpha(m),\cdots, \alpha ((n-1)m)).$$

 \noindent Denote by $\ccdnm{n}{m}(P)$ the smallest positive integer $r'$ such that $P^n$ can be cover by $r'$ open sets $\{Q_i\}_{i=1}^{r'}$ with a section $s_i : Q_i \rightarrow P^{J_{(n-1)m}}$ of $q'_{n, m}$ for each $i$. We have a analogue of Lemma \ref{ccnm}. 
 
 \begin{lemma}\label{lccdmn}
 For any $m \geq 0, n\geq 2$, it holds that $\ccdnm{n}{m+1}(P) \leq \ccdnm{n}{m}(P)$.
 \end{lemma}
 
 \begin{proof}
Let $q'_{n, m}$ has a section on an open set $Q_i$ of $P^n$. We define a order preserving map $$R : J_{(n-1)(m+1)} \rightarrow J_{(n-1)m}$$ which maps $i(m+1)-1$, $i(m+1)$ and $i(m+1)+1$ to $im$, if $i$ is odd otherwise it is linear. This will induce a map $R^{*} : P^{J_{(n-1)m}} \rightarrow P^{J_{(n-1)(m+1)}}$ such that the following diagram commutes:
 \begin{center}
 $ \begin{tikzcd}
P^{J_{(n-1)m}} \arrow{rr}{R^{*}} \arrow[swap]{dr}{q'_{n, m}}& &P^{J_{(n-1)(m+1)}} \arrow{dl}{q'_{n, m+1}}\\
& P^n & 
\end{tikzcd}$
 \end{center}
 The composition $R^* \circ s_i : Q_i \rightarrow P^{{(n-1)(m+1)}}$ is a section of $q'_{n, m+1}$ for each $i$. Thus, $\ccdnm{n}{m+1}(P) \leq \ccdnm{n}{m}(P)$.
 \end{proof} 
 
 \begin{definition}\label{dccn'}
 We define $\ccdn{n}(P)$ is the minimum of the $\ccdnm{n}{m}(P)$

 $$ \ccdn{n}(P) = \min_{m \geq 1}\{ \ccdnm{n}{m}(P)\} = \lim_{m \rightarrow \infty}\ccdnm{n}{m}(P).$$ 

 \end{definition}
 
 We now prove that the two Definitions \ref{dccn} and  \ref{dccn'} are equivalent.
 
 \begin{theorem}\label{thm}
 For any finite space $P$,  $\ccdn{n}(P) = \ccn{n}(P) $.
 \end{theorem}
 \begin{proof}
 Assume that $\ccn{n}(P) = r$. Take an open cover $\{Q_i\}_{i=1}^r$ of $P^n$ with order preserving section $s_i : Q_i \rightarrow P^{J_{n, m}}$ of $q_{n, m}$, for some $m$ and each $i$. Since $\ccnm{n}{m}$ is decreasing with respect to $m$ by Lemma \ref{ccnm} , we can assume that $m$ is even. Define an order preserving map $f : J_{(n-1)2m} \rightarrow J_{n, m}$ by sending each element of $J_{(n-1)2m} $  to the corresponding element of the following path of  $J_{n, m}$:
 \begin{center}
$$m_1 \leq m_1 - 1 \geq \cdots\leq 1_1 \geq 0 \leq 1_2 \geq \cdots\leq m_2-1 \geq m_2 \leq m_2-1 \geq \cdots \leq 1_2 \geq 0 \leq 1_3 $$

$$\geq \cdots \leq m_3-1 \geq m_3 \leq m_3-1 \geq \cdots \leq 1_3 \geq 0 \leq \cdots \geq 0 \leq 1_n \geq \cdots \leq m_n-1 \geq m_n.$$
 \end{center}
 
 \noindent This map induces $f^{*} : P^{J_{n, m}} \rightarrow P^{J_{(n-1)2m}}$ such that the following triangle commutes.
 \begin{center}
 $ \begin{tikzcd}
P^{J_{n, m}} \arrow{rr}{f^{*}} \arrow[swap]{dr}{q_{n, m}}& &P^{J_{(n-1)2m}} \arrow{dl}{q'_{n, 2m}}\\
& P^n & 
\end{tikzcd}$
 \end{center}
 So the composition map $f^{*} \circ s_i : Q_i \rightarrow P^{J_{(n-1)2m}}$ is a section of $q'_{n, 2m}$ on $Q_i$ for each $i$. Thus $\ccdnm{n}{2m}(P) \leq r$ i.e. $\ccdn{n}(P) \leq r = \ccn{n}(P)$.
 
Conversely, assume that $\ccdn{n}(P) = r'.$ Then there are sections $s_i : Q_i \rightarrow P^{J_{(n-1)m}}$ of $q'_{n, m}$ on open subsets $Q_i$ of $P^n$ coving it, for some $m$, ~$1\leq i \leq r'$. We can take $m$ as a multiple of four, by Lemma \ref{lccdmn}. Let $k = \frac{(n-1)m}{2}$ so that $k$ is even. Define an order preserving retraction map $g : J_{n, k} \rightarrow J_{(n-1)m}$ by sending $0 \leq 1_j \geq 2_j \leq \cdots \geq k_j $ to ($1\leq j \leq n$):

$$ k \leq k-1 \geq k-2 \leq \cdots \geq (j-1)m \leq (j-1)m \geq (j-1)m \leq \cdots \geq (j-1)m,~~\text{if} ~~j \leq \frac{n}{2}$$ 
$$k \leq k+1 \geq k+2 \leq \cdots \geq (j-1)m \leq (j-1)m \geq (j-1)m \leq \cdots \geq (j-1)m ~~~ \text{if} ~~j > \frac{n}{2}.$$

 \noindent As in previous case the map $g$ map induces $g^{*} : P^{J_{(n-1)m}} \rightarrow P^{J_{n ,k}}$ such that the following triangle commutes.
 \begin{center}
 $ \begin{tikzcd}
P^{J_{(n-1)m}} \arrow{rr}{g^{*}} \arrow[swap]{dr}{q'_{n, m}}& &P^{J_{n ,k}} \arrow{dl}{q_{n, k}}\\
& P^n & 
\end{tikzcd}$
 \end{center}
 So the composition map $g^{*} \circ s_i : Q_i \rightarrow P^{J_{k, n}}$ is a section of $q_{n, k}$ for each $1\leq i \leq r'$. Thus $\ccnm{n}{k}(P)\leq  r'$ i.e. $\ccn{n}(P) \leq \ccdn{n}(P)$. Hence $\ccdn{n}(P) = \ccn{n}(P)$.
 \end{proof}
 
  To abuse the notation we will only use $\ccn{n}(P)$ for both the descriptions.  Similar to the topological setting, we have an upper bound of $\ccn{n}(P)$ in terms of $L\text{-}S$ category of $P^n$. 
 \begin{lemma}
 It holds that $\ccn{n}(P) \leq cat(P^n)$ for all $n \geq 2$.
 \end{lemma}
 \begin{proof}
 Let $cat(P^n) = r$. Then there exists contractible open cover  $\{Q_i\}_{i=1}^{r}$ of $P^n$. For a fixed $i,$ let $Q_i$ be contractible to $(x_1, x_2,\cdots,x_n) \in P^n$. Since $P$ is connected there exists a positive integer $k$ and a map $$\gamma : J_{n, k} \rightarrow P \text{ such that } q_{ n,k} (\gamma) = (x_1, x_2,\cdots,x_n).$$ Now let $(p_1, p_2,\cdots,p_n) \in Q_i$ be an arbitrary element. Choose a contracting homotopy of $Q_i$ in $P^n$, $H : Q_i \times J_l \rightarrow P^n$ such that $H(p_1, p_2,\cdots,p_n, 0) = (x_1, x_2,\cdots,x_n)$ and $H(p_1, p_2,\cdots,p_n, l) = (p_1, p_2,\cdots,p_n)$. Applying exponential law we get projection maps $H_j : Q_i \rightarrow P^{J_l}$ such that $H_j(p_1, p_2,\cdots,p_n)(l) = p_j$ and $H_j(p_1, p_2,\cdots,p_n)(0) = x_j$ for $j = 1, 2,\cdots,n$. Assume $\gamma_j: J_k \to P$ be $j$-th component of $\gamma$ which is obtained by composing $\gamma$ with $n$-different inclusions of $J_k$ inside $J_{n,k}$. Set $\lambda_j = \gamma_j\ast H_j$, concatenation of $H_j$ and $\gamma_j$. Let $\lambda : J_{n, k+l} \rightarrow P$ or $\lambda : J_{n, k+l+1} \rightarrow P$ according to $l$ even or odd, is the map whose components are $\lambda_i$'s. Define $s_i(p_i, p_2,\cdots,p_n) = \lambda$, whose projection nothing but $(p_i, p_2,\cdots,p_n)$. Set $m_i = k+l$ or $k+l+1$ according to $l$ even or odd and $m = \max \{m_1, m_2,\cdots,m_r\}$. Then $\ccnm{n}{m}(P) \leq r$ and therefore $\ccn{n}(P) \leq cat(P^n)$ for all $n \geq 2$.
 \end{proof}
 
Now we will prove the equality between $\ccn{n}(P)$ and $TC_n(P)$.  For this first we recall order complex $\mathcal{K}(P)$ of a finite space $P$. 
 
 \begin{definition}\label{oc}
 The \emph{order complex} $\mathcal{K}(P)$ of a finite space $P$ is a simplicial complex whose $n$-simplices are linearly ordered subsets of $P$. Its geometric realization is denoted by $\mid \mathcal{K}(P) \mid$.
 
 \end{definition}
 \begin{example}\label{exsn}
Let $\mathbb{S}^m$ denote the finite space consisting of $2m+2$ points $$\{e_+^0, e_-^0,\cdots, e_+^m, e_-^m\},$$ with the partial order defined by $e_p^k \leq e_q^l$ if $k \leq l$ and $p, q \in \{+, -\}$. The realization of the order complex $\mid \mathcal{K}(\mathbb{S}^m) \mid$ is homeomorphic to the sphere $S^m$ with dimension $m$. 
\end{example}

 \begin{theorem}\label{tctp}
For any finite space $P$, it holds that $\ccn{n}(P) = TC_n(P)$, $n \geq 2$.
\end{theorem}
\begin{proof}
Let us first show that $TC_n(P) \geq \ccn{n}(P)$. Assume that $TC_n(P) = r$ with an open cover $\{Q_i\}_{i=1}^n$ of $P^n$ and a continuous section $Q_i \rightarrow P^I$ of $e$ for each $i$. This induces a map $f_1 : I \rightarrow P^{Q_i}$ by the exponential law. Hence we obtain a order preserving map $f_2 : J_s \rightarrow P^{Q_i}$ for some $s \geq 0$ by the homotopy theory of finite spaces. Now we can construct a order preserving map $\phi : J_{(n-1)m} \rightarrow J_s$ for some $m \geq 0$, such a way that $f_1(\frac{k}{n-1}) = f_2\circ\phi (km),~~ k = 0, 1,\cdots , n-1$. So we have a combinatorial section $Q_i \rightarrow P^{J_{(n-1)m}}$ of $q'_{n, m}$. Hence $\ccn{n}(P) \leq r$.

  For the other inequality, assume $\ccn{n}(P) = r$ with open cover $\{Q_i\}_{i=1}^n$ of $P^n$ and a continuous section $s_i : Q_i \rightarrow P^{J_{(n-1)m}}$ for some $m \geq 0$ of $q'_{n, m}$, for all $i$.

 \noindent Let $\alpha : [0, (n-1)m] \simeq \mid \mathcal{K}(J_{(n-1)m}) \mid \rightarrow J_{(n-1)m}$ be denote the continuous map (see \cite{mc}):
\begin{center}
$\alpha (t) = \begin{cases}
2k-1 & \mbox{ if } ~~ t = 2k-1\\
2k &\mbox{if } ~~ 2k-1 < t < 2k+1.
 \end{cases}$
\end{center}

Note that, $\alpha (jm) = jm, j = 0, 1,\cdots,(n-1)$. Let $\beta : I \rightarrow J_{(n-1)m}$ denote the composition of $\alpha$ and the homeomorphism  $I = [0, 1] \simeq [0, (n-1)m]$, given by $t \rightarrow (n-1)mt$. This induces $\beta^* : P^{J_{(n-1)m}} \rightarrow P^I$ such that the diagram commutes:
\begin{center}
$ \begin{tikzcd}
P^{J_{(n-1)m}} \arrow{rr}{\beta^{*}} \arrow[swap]{dr}{q'_{n, m}}& &P^{I} \arrow{dl}{e_n}\\
& P^n & 
\end{tikzcd}$
\end{center}
The composition $\beta^* \circ s_i : Q_i \rightarrow P^I$ is a continuous section for the fibration $e_n$. So $TC_n(P) \leq r$.
Thus $\ccn{n}(P) = TC_n(P)$, $n \geq 2$.
\end{proof}

\begin{corollary}\label{cor1}
\begin{enumerate}[(a)]
\item For any finite space $P$ we have $ \ccn{n}(P) \leq \ccn{n+1}(P)$ , $n \geq 2$. 

\item A finite space $P$ is contractible if and only if $\ccn{n}(P) = 1$ for any $n \geq 2$.

\item The $\ccn{n}$ is homotopy invariant, i.e. $\ccn{n}(P) = \ccn{n}(Q)$ if two spaces $P$ and $Q$ are homotopy equivalent.

\item For any finite space $P$, the following inequalities hold:

$$cat(P) \leq \ccn{n}(P) \leq cat(P^n) \leq cat(P)^n.$$

\end{enumerate}
\end{corollary}
\begin{proof}
This follows from above Theorem \ref{tctp} and the corresponding inequalities about $TC_n(P)$. 

\end{proof}

\begin{example}
Let $\mathbb{S}^m$ denote the finite space consisting  as in Example \ref{exsn}. Then

$$\ccn{n}(\mathbb{S}^m) = cat((\mathbb{S}^m)^n) = cat(\mathbb{S}^m)^n = 2^n \text{ for any } m \geq 1 \text{ and } n \geq 2.$$

\end{example}
\begin{proof}
We know that $cat(\mathbb{S}^m) = 2$ (see \cite[Example 3.7]{tana}) for any $m \geq 1$. So by Corollary \ref{cor1}, $\ccn{n}(\mathbb{S}^m) \leq cat(\mathbb{S}^m)^n = 2^n$. If $\ccn{n}(\mathbb{S}^m) < 2^n$, then there is an open set of $(\mathbb{S}^m)^n$ containing at least two distinguished maximal points of $(\mathbb{S}^m)^n$. This yields a contractible open set in $\mathbb{S}^m$ containing $e_+^m$ and $e_-^m$, that will be nothing but the entire space $\mathbb{S}^m$, which is not contractible. The contradiction implies that $\ccn{n}(\mathbb{S}^m) = 2^n$. Thus, $\ccn{n}(\mathbb{S}^m) = cat((\mathbb{S}^m)^n) = cat(\mathbb{S}^m)^n = 2^n$ for any $m \geq 1$ and $n \geq 2$.
\end{proof}

\end{section}
\begin{section}{Higher combinatorial complexity with barycentric subdivision}

The combinatorial complexity $\cckn{k}{n}(P)$  does not capture the topological complexity of the naturally associated simplicial complex $\mathcal{K}(P)$. From the example of the $\mathbb{S}^m$ we see that $\cckn{k}{n}(\mathbb{S}^m)$ is much higher than $TC_n({S}^m)$. To remedy the situation, we refine the definition of $\cckn{k}{n}(P)$ using barycentric subdivision of $P$. 
%Tanaka defined $k$-th barycentric subdivisional combinatorial complexity (cf. \cite[Definition 4.5]{tana}) for a finite space $P$. Here we simply generalize for $n$-th combinatorial complexity $\cckn{k}{n}(P)$, where $\cckn{k}{2}(P)$ is nothing but that of Tanaka. 
We first recall barycentric subdivision of a finite space $P$. Then we define higher combinatorial complexity with barycentric subdivision and show it is equal to the topological complexity of $|\mathcal{K}(P)|$.
\begin{definition}
   The \emph{barycentric subdivision} $\sd(P)$ of $P$ is defined as the face poset $\mathcal{X}(\mathcal{K}(P))$ of the order complex $\mathcal{K}(P)$ (\ref{oc}). 
\end{definition}
 \noindent Let $\tau_P : \sd(P) \rightarrow P$ be the map sending $p_0 \leq p_1 \leq \cdots \leq p_n$ to the last element $p_n$. This is a weak homotopy equivalence, and the induced simplicial map $\mathcal{K}(\tau_P) : \mathcal{K}(\sd(P)) = \sd(\mathcal{K}(P)) \rightarrow \mathcal{K}(P)$ is a simplicial approximation of the identity on $\mid \mathcal{K}(P) \mid$ (see \cite{har}). For $k \geq 0$, we denote by  $\tau_P^k : \sd^k(P) \rightarrow P$ the composition
   \begin{center}
      $ \sd^k(P) \xrightarrow {\tau_{\sd^{k-1}(P)}} \sd^{k-1}(P) \xrightarrow {\tau_{\sd^{k-2}(P)}} \cdots   \xrightarrow {\tau_{\sd(P)}} \sd(P) \xrightarrow {\tau_P} P$.
   \end{center}
  
   \begin{definition}
Let $k \geq 0$. We define $\cckn{k}{n}(P)$ as the smallest nonnegative integer $r$ such that there exist an open cover $\{Q_i\}_{i=1}^r$ of $\sd^k (P^n)$ and an positive integer $m$, with a map $s_i : Q_i \rightarrow P^{J_{n, m}}$ such that $q_{n, m} \circ s_{i} =  \tau_{P^n}^k$ on $Q_i$ for each $i$. If no such $r$ exists, then we define $\cckn{k}{n}(P)$ to be $\infty$.
\end{definition}

  If we define $\ccdkn{k}{n}(P)$ as taking $P^{J_{(n-1)m}}$ instead of $P^{J_{n, m}}$ and $q'_{n, m}$ of $q_{n ,m}$ then we get same positive integer. Obviously, $\ccn{n}(P) = \cckn{0}{n}(P)$ by the definition above.
 
\begin{lemma}
For any finite space $P$ and $n\geq 2$, we have: 
\begin{enumerate}[(a)]
 \item $\ccdkn{k}{n}(P) = \cckn{k}{n}(P) $.
 \item $\cckn{k+1}{n}(P) \leq \cckn{k}{n}. $ 
 \end{enumerate}
 \end{lemma}
 \begin{proof}
 (a)This proof is similar as the proof of the Theorem (\ref{thm}).
 
 (b) The result for $n = 2$ was proved in \cite[Lemma 4.3]{tana}. In similar way we can prove this result.
 \end{proof}
 
 \begin{definition} 
 We define the \emph{$n$-th combinatorial complexity} $\cckn{\infty}{n}(P)$ of $P$ to be

  $$\cckn{\infty}{n}(P) = \lim_{k \rightarrow \infty} \cckn{k}{n}(P) = \min_{k \geq 0} \{\cckn{k}{n}(P)\}. $$

\end{definition}

Now we relate $\cckn{\infty}{n}(P)$ to the $n$-th topological complexity $TC_n(\mid \mathcal{K} (P) \mid )$ of geometric realization of the order complex of $P$. For this we need the following lemma.  For $k \geq 0$, let $\rho_j : \sd^k(P^n) \rightarrow P$ denote the composition of $\tau^k_{P^n} : \sd^k(P^n) \rightarrow P^n$ and the $j$-th projection for $j = 1, 2,\cdots,n$.

\begin{lemma}\label{lemma}
With notations as above, $\cckn{\infty}{n}(P) \leq r$ if and only if there exist $k \geq 0$ and an open cover $\{Q_i\}_{i = 1}^r$ of $\sd^k(P^n)$ such that $\rho_1, \rho_2,\cdots,\rho_n : Q_i \rightarrow P$ are in same homotopy class of maps.
\end{lemma}
\begin{proof}
Let us assume that $\cckn{\infty}{n}(P) \leq r$. Then for some $k \geq 0$ there exist an open cover $\{Q_i\}_{i=1}^r$ of $\sd^k(P^n)$  and a map $s_i : Q_i \rightarrow  P^{J_{(n-1)m}}$  such that $q'_{n, m} \circ s_i = \tau^k_{P^n}$ on $Q_i$ for each $i$.
\begin{center}
$\begin{tikzcd}
Q_i \arrow[hookrightarrow]{d} \arrow{r}{s_i}
& P^{J_{(n-1)m}} \arrow{d}{q'_{n, m}} \\
\sd^k(P^n) \arrow{r}[swap]{\tau^k_{P^n}}
& P^n
\end{tikzcd}$
\end{center}
We define a homotopy $H : Q_i \times J_{(n-1)m} \rightarrow P$ by
$$H(p_1, p_2,\cdots,p_n, x) = s_i (p_1, p_2,\cdots ,p_n)(x),$$ where  $(p_1, p_2,\cdots ,p_n) \in Q_i , ~~x\in J_{(n-1)m}.$ Then 
$$H(p_1, p_2,\cdots,p_n, (j-1)m) = s_i (p_1, p_2,\cdots ,p_n)((j-1)m)= \rho_j(p_1, p_2,\cdots,p_n), $$
for $j \in \{1, 2,\cdots , n\}$. This shows that the maps $\rho_1, \rho_2,\cdots,\rho_n : Q_i \rightarrow P$ are in same homotopy class of maps.

Conversely, assume that $\{Q_i\}_{i = 1}^r$  is an open cover of $\sd^k(P^n)$ such that  $\rho_1, \rho_2,\cdots,\rho_n : Q_i \rightarrow P$ are in same homotopy class of maps for each $i$. Then there exist homotopies $H_j : Q_i \times J_m \rightarrow P$ for some $m$ and $j = 1, 2,\cdots, n-1$, such that $H_j$ is a homotopy between $\rho_{j}$ and $\rho_{j+1}$. Define $s_i : Q_i \rightarrow P^{J_{(n-1)m}}$ by 

 $$s_i (p_1, p_2,\cdots,p_n)(x) =[ H_1*H_2 *\cdots*H_{n-1}](p_1, p_2,\cdots,p_n, x),$$
where $\ast$ denotes concatenation and  $x\in J_{(n-1)m}$. So we have $$q'_{n, m} \circ s_i(p_1, p_2,\cdots,p_n)(j) = \rho_j (p_1, p_2,\cdots,p_n).$$ Thus $q'_{n, m} \circ s_i(p_1, p_2,\cdots,p_n) = \tau^{k}_{P^n} (p_1, p_2,\cdots,p_n),~~ 1\leq j\leq n$. This gives a section over each $Q_i$ for $i =1, \cdots, r.$ Hence $\cckn{\infty}{n}(P) \leq r.$

\end{proof}

We now prove the main result of the section, generalising \cite[Theorem 4.9  ]{tana}

\begin{theorem}\label{tsc}
For any finite space $P$, we have $\cckn{\infty}{n}(P) = SC_n(\mathcal{K}(P))$, $n \geq 2$.
\end{theorem}
\begin{proof}
Assume that $\cckn{\infty}{n}(P) = r$. By the Lemma (\ref{lemma}) there exists $k \geq 0$ and an open cover $\{Q_i\}_{i=1}^{r}$ of $\sd^k(P^n)$ such that $\rho_1, \rho_2,\cdots,\rho_n : Q_i \rightarrow P$ are in same homotopy class of maps. Using  \cite[Proposition 4.11]{bar} we can say that the maps $$\mathcal{K}(\rho_1), \mathcal{K}(\rho_2),\cdots,\mathcal{K}(\rho_n) : \mathcal{K}(Q_i) \rightarrow \mathcal{K}(P)$$ lie in same contiguity class. The subcomplex $\mathcal{K}(Q_i)$ form a cover of $\mathcal{K}(\sd^k(P^n)) = \sd^k(\mathcal{K}(P^n))$ and $\mathcal{K}(\rho_j) = \mathcal{K}(pr_j \circ \tau^k_{P^n}) = \pi_j$, for $j = 1, 2,\cdots,n$. So, $SC_n^k(\mathcal{K}(P)) \leq r$ and then $SC_n(\mathcal{K}(P)) \leq r$.

 Conversely, assume that $SC_n(\mathcal{K}(P)) = r$. Then $SC_n^k(\mathcal{K}(P)) = r$ for some $k \geq 0$. Let $\{L_i\}_{i=1}^r$ be a covering of $\sd^k(\mathcal{K}(P^n))$ and the restriction $\pi_1, \pi_2,\cdots,\pi_n : L_i \rightarrow \mathcal{K}(P)$ lie in same contiguity class for each $i$. The  \cite[Proposition 4.12]{bar} implies that $\mathcal{X}(\pi_1), \mathcal{X}(\pi_2),\cdots,\mathcal{X}(\pi_n) : \mathcal{X}(L_i) \rightarrow \mathcal{X}(\mathcal{K}(P)) = \sd(P)$ are in same homotopy class of maps. The subsets $\mathcal{X}(L_i)$ form an open cover of $\mathcal{X}(\sd^k(\mathcal{K}(P^n))) = \sd^{k+1}(P^n)$. The naturality of $\tau$ makes the following diagram commute :
\begin{center}
$\begin{tikzcd}  
 \sd^{k+1}(P^n)\arrow{r}{\sd(\tau^k_{P^n})}\arrow{d}{\tau_{\sd^k(P^n)}} 
  & \sd(P^n)\arrow{r}{\sd(pr_j)}\arrow{d}{\tau_{P^n}} 
  & \sd(P)\arrow{d}{\tau_P} 
   \\
   \sd^k(P^n)\arrow{r}[swap]{\tau^k_{P^n}} 
  & P^n\arrow{r}[swap]{pr_j} 
  & P
\end{tikzcd}$
  \end{center}
  
 Also we have
 $$\tau_P \circ \mathcal{X}(\pi_j) = \tau_P \circ \mathcal{X}(\mathcal{K}(pr_j \circ \tau^k_{P^n}))
 = \tau_P \circ \sd(pr_j \circ \tau^k_{P^n}) = \tau_P \circ \sd(pr_j) \circ \sd(\tau^k_{P^n})$$
\begin{center}
$ \ \ \ \ \ \ \ \ \ \ \ \ \ \ \ \ \ \ \ \ \ \ \ \ \ \ \ \ \ \ \ \ \ \ = pr_j \circ \tau_{P^n}^{k+1} = \rho_j$ \ \ \ \ \ for $j = 1, 2,\cdots,n$.
\end{center}

So, $\cckn{k+1}{n}(P) \leq r$ and then $\cckn{\infty}{n}(P) \leq r$.
Thus $\cckn{\infty}{n}(P) = SC_n(\mathcal{K}(P))$, $n \geq 2$.
\end{proof}
Combining Theorem \ref{tsc} and Theorem \ref{tst} we have the following corollary.
\begin{corollary}
For  any finite space $P$, we have $\cckn{\infty}{n}(P) = TC_n(\mid \mathcal{K}(P) \mid )$, $n \geq 2$. 
\end{corollary}

\end{section}

\end{document}